\documentclass[12pt,epsfig,amsfonts]{amsart} 
\usepackage{amsmath,amsthm,amssymb,amscd,epsfig,color}
\usepackage{graphicx}
\usepackage{mathrsfs}
\usepackage{ulem}

\def\qed{\hfill $\Box$}

\newcommand{\confrac}[2]{%
  \frac{\displaystyle{%
    \strut\hfill{#1}\hfill\;\vrule}}%
      {\displaystyle{%
       \strut\vrule\;\hfill{#2}\hfill}}}%

\theoremstyle{plain}
\newtheorem{thm}{Theorem}[section]
\newtheorem{lem}[thm]{Lemma}
\newtheorem{cor}[thm]{Corollary}

\newtheorem{pro}[thm]{Proposition}
\newtheorem*{syuA}{Main Theorem}

\theoremstyle{definition}

\newtheorem*{prf*}{Proof}

\newtheorem*{pf*}{}
\newtheorem*{lem*}{LemmaA}
\newtheorem*{lm*}{LemmaB}
\newtheorem*{clm*}{Claim}
\newtheorem*{stra*}{Strategy for the proof of main result A}

\def\g2{l\ge2}

\def\m2l{\mathcal{L}_2}

\makeatletter
\@namedef{subjclassname@2020}{%
  \textup{2020} Mathematics Subject Classification}
\makeatother
\title[Sets with restricted, slowly growing partial quotients]{Hausdorff dimension of sets with\\ restricted, slowly growing partial quotients\\
in semi-regular continued fractions}

\author{Yuto Nakajima and Hiroki Takahasi}

\address{Department of Mathematics,
Keio University, Yokohama,
223-8522, JAPAN}
\email{nakajimayuto@math.keio.ac.jp}
\address{Department of Mathematics,
Keio University, Yokohama,
223-8522, JAPAN} 
\email{hiroki@math.keio.ac.jp}


\begin{document}
\maketitle
\begin{abstract}
We determine the Hausdorff dimension of sets of irrationals in $(0,1)$ whose partial quotients in semi-regular continued fractions obey  certain restrictions and growth conditions. This result substantially generalizes
that of the second author [Proc. Amer. Math. Soc. {\bf 151} (2023), 3645--3653] and the solution of Hirst's conjecture [B.-W. Wang and J. Wu, Bull. London Math. Soc. {\bf 40} (2008), 18--22], both previously obtained for the regular continued fraction.
To prove the result, we 
construct non-autonomous iterated function systems well-adapted to the given restrictions and growth conditions on partial quotients, estimate the associated pressure functions, and then apply Bowen's formula.
\end{abstract}
\section{Introduction}
   Each irrational number $x$ in the set $\mathbb I=(0,1)\setminus\mathbb Q$
   has a unique
 regular continued fraction (RCF) expansion 
 $x=
 1/(a_1+1/(a_2+\cdots))$, where $a_n$, $n\geq1$ belongs to the set $\mathbb N$ of positive integers.
For a typical number in $\mathbb I$ in the sense of the Lebesgue measure, 
 Khinchin \cite{Khi64} proved that
 the frequency with which
 each integer $k\in\mathbb N$ appears as 
 its RCF partial quotients is $\frac{1}{\log2}\log\frac{(k+1)^2}{k(k+2)}$.
 Although Khinchin did not use ergodic theory in his original proof, his result is a consequence of Birkhoff's theorem applied to the Gauss map 
 leaving the Gauss measure $\frac{1}{\log2}\frac{dx}{1+x}$ invariant and ergodic.

 Numbers in $\mathbb I$ whose RCF partial quotients behave very differently from the Lebesgue typical ones are not negligible in the sense of the Hausdorff dimension.
 This fact lies at the basis of the fractal dimension theory of continued fractions.
A pioneering result in the creation of the theory is due to Jarn\'ik \cite{Jar28}, who proved that the set 
\[\left\{x\in\mathbb I\colon \{a_n(x)\}\text{ is bounded}\right\},\]
i.e., the set of badly approximable numbers \cite[Theorem~23]{Khi64},
is of Hausdorff dimension $1$.
Later on,
Good \cite{Goo41} proved that the set
\[\left\{x\in\mathbb I\colon \lim_{n\to\infty}a_n(x)=\infty\right\}\]
is of Hausdorff dimension $1/2$.
Various extensions and refinements of this Good's theorem have been made (see \cite{CWW13,Cus90,FLWJ09,FWLT97,Hir70,Hir73,JaeKes10,JorRam12,Luc97,Moo92,Mun11,Ram85,T,WanWu08} for example)
which determine fractal dimensions of sets of numbers in $\mathbb I$ whose RCF partial quotients obey certain
restrictions or growth conditions.

Among others, Hirst \cite{Hir73} proved results analogous to that of Good \cite{Goo41} in the case where $a_n$
 is restricted to belonging to some subsets of $\mathbb N$.
 For an infinite subset $B$ of $\mathbb N$,
 define the {\it exponent of convergence} by
 \[\tau(B)=\inf\left\{s\geq0\colon\sum_{
 k\in B}k^{-s}<\infty\right\}.\]
 In \cite{Hir73},
 Hirst considered the set
 \[E(B)=\left\{x\in\mathbb I\colon a_n(x)\in B\text{ for all } n\in\mathbb N\text{ and } \lim_{n\to\infty}a_n(x)=\infty\right\},\]
 and conjectured that 
 \begin{equation}\label{Hirst}\dim_{\rm H}E(B)=\frac{\tau(B)}{2}\ \text{ for any }B,\end{equation}
where $\dim_{\rm H}$ denotes the Hausdorff dimension.
He treated
 the special case $B=\{k^b\}_{k\in\mathbb N}$, $b$
 a positive integer
 to support his conjecture (see \cite[Theorem~3]{Hir73}).
 Cusick \cite[Theorem~1]{Cus90} proved that the equality holds in the case $B$ is not too sparse and satisfies what he called the density assumption.
 35 years later than the appearance of Hirst's paper \cite{Hir73},
  his conjecture was confirmed
   by
 Wang and Wu \cite[Theorem~1.1]{WanWu08}. 
 This solution was extended in
  \cite{CWW13} 
 to certain infinite iterated function systems 
  satisfying the so-called open set condition \cite{Fal14}.

In this context, 
subsets of $E(B)$ of the form
 \[F(B,f)=\{x\in E(B)\colon a_n(x)\geq f(n)\text{ for all } n\in\mathbb N\},\]
 and
 \[G(B,f)=\left\{x\in E(B)\colon a_n(x)\leq f(n)\text{ for all }n\in\mathbb N\right\},\]
 where $f\colon\mathbb{N}\rightarrow \mathbb N$ is a function satisfying $\lim_{n\to \infty}f(n)=\infty$ 
  are also relevant to consider. 
  In \cite[p.227]{Hir73}, Hirst conjectured  the equality
$\dim_{\rm H}F(B,f)=\tau(B)/2$ no matter how rapidly $f$ grows, and 
 Cusick later showed the case $B=\mathbb N$, $f(n)=2^{2^{2^n}}$ 
as a counterexample (see \cite[Lemma~3]{Cus90}).
Nowadays it is known that $\dim_{\rm H}F(B,f)$ can drop from $\tau(B)/2$ and even become $0$ when $f$ grows very rapidly \cite{CWW13,Cus90,FWLT97,JorRam12,Luc97,Moo92,WanWu08II}, whereas 
  ${\rm dim}_{\rm H}G(B,f)$ does not drop at all no matter how slowly $f$ grows to infinity \cite{T}: Interestingly we have
  \begin{equation}\label{Takahasi}
  \dim_{\rm H}G(B,f)=\frac{\tau(B)}{2}\ \text{ for any $B$ and any $f$ with $f(\mathbb N)\subset[\min B,\infty)$.}
  \end{equation}

The aim of this paper is to 
 extend
the dimension results \eqref{Hirst}, \eqref{Takahasi} on the RCF to the semi-regular continued fraction (SRCF).
For each
$\sigma=\{\sigma_n\}_{n=1}^\infty\in\{-1,1\}^\mathbb N$, any $x\in\mathbb I$ has a unique SRCF expansion of the form 
\begin{equation}\label{CF-general}
x=\cfrac{1}{a_{\sigma,1}(x)+\cfrac{\sigma_1}{a_{\sigma,2}(x)+\cfrac{\sigma_2 }{a_{\sigma,3}(x)+\ddots}}}=
\confrac{1 }{a_{\sigma,1}(x)} + \confrac{\sigma_1 }{a_{\sigma,2}(x)}  + \confrac{\sigma_{2} }{a_{\sigma,3}(x) }+\cdots,
\end{equation}
where 
 $a_{\sigma,n}=a_{\sigma,n}(x)$ are positive integers such that $\sigma_n+a_{\sigma,n}\geq1$
for all $n\geq1$. 
This means that the
finite truncation 
\[\cfrac{1}{a_{\sigma,1}+\cfrac{\sigma_1 }{a_{\sigma,2}+\cdots+\cfrac{\ddots}{\displaystyle{\frac{\sigma_{n-1}}{a_{\sigma,n}}}}}}=\confrac{1 }{a_{\sigma,1}} + \confrac{\sigma_1 }{a_{\sigma,2}}  +\cdots+ \confrac{\sigma_{n-1} }{a_{\sigma,n} }\]
converges to $x$ as $n\to\infty$.

The SRCF, introduced by Perron \cite{Per50} in a slightly more general form,
  includes several continued fractions as particular cases. 
If $\sigma_n=1$ for all $n\geq1$ we obtain the RCF, and if
 $\sigma_n=-1$ for all $n\geq1$ we obtain the backward continued fraction (BCF). 
 For other examples and motivations for the SRCF, see
 \cite{Bos87,DK99,IosKra02,Kra91,Per50}.
 Any SRCF can be transformed into a regular one, as shown in \cite{Per50}. The converse is also true \cite{DK99}.


 For the SRCF for which $\sigma\in\{-1,1\}^\mathbb N$ is not eventually constant,
 there is no single interval transformation that generates the expansion. Therefore, developing 
a dimension theory on fractal sets determined by the SRCF requires novel 
approaches.

\subsection{Statements of the results}
To obtain results analogous to \eqref{Hirst} and \eqref{Takahasi},
for each
$\sigma=\{\sigma_n\}_{n=1}^\infty\in\{-1,1\}^\mathbb N$ 
we consider sets 
  \[E_\sigma(B)=\left\{x\in\mathbb I\colon a_{\sigma,n}(x)\in B\text{ for all } n\in\mathbb N\text{ and } \lim_{n\to \infty}a_{\sigma, n}(x)=\infty\right\},\]
  and
    \[G_{\sigma}(B, f)=\left\{x\in E_\sigma(B)\colon a_{\sigma, n}(x)\le f(n)\ \mbox{for all } n\in \mathbb{N} \right\}.\]
    Clearly, we have
$\dim_{\rm H}E_\sigma(B)\geq\dim_{\rm H}G_{\sigma}(B, f)$.
\begin{syuA}
\label{renbun}
For any $\sigma\in \{-1, 1\}^{\mathbb{N}},$ any infinite subset $B$ of $\mathbb{N}$ and any function  $f:\mathbb{N}\rightarrow [\min B, \infty)$ such that $\lim_{n\to \infty}f(n)=\infty,$ we have
\[\dim_{\rm H}E_\sigma(B)=\dim_{\rm H}G_{\sigma}(B, f)
=\frac{\tau(B)}{2}.\]
\end{syuA}





Analogues of
Good's theorem \cite{Goo41} 
on the Hausdorff dimension of sets with divergent partial quotients 
were obtained for other series expansions of numbers, see
\cite{Rob20,Mun11}.
Taking $B=\mathbb N$
in the Main Theorem yields an extension of Good's theorem to the SRCF.
\begin{cor}
For any $\sigma\in \{-1, 1\}^{\mathbb{N}}$, we have \[\dim_{\rm H}\left\{x\in \mathbb{I}\colon \lim_{n\to \infty}a_{\sigma, n}(x)=\infty \right\}=\frac{1}{2}.\]\end{cor}


\subsection{The method of proof}
The equality \eqref{Takahasi} 
was proved in \cite{T} by combining
the general upper bound $\dim_{\rm H} E(B)\leq\tau(B)/2$ in \cite[Corollary~1]{Hir73} with
a new lower bound $\dim_{\rm H} G(B,f)\geq\tau(B)/2$, 
obtained in \cite{T} by constructing a fractal subset of $G(B,f)$ and estimating its Hausdorff dimension from below.
The argument in \cite{T} remains valid 
in the case where $\sigma\in\{-1,1\}^\mathbb N$ is eventually constant.
In order to treat all other
 $\sigma$,
 new ingredients are necessary as we develop below.

To obtain an upper bound $\dim_{\rm H}E_\sigma(B)\leq\tau(B)/2$,
we slightly modify the upper bound in \cite[Corollary~1]{Hir73} which relies on the well-known formula for fundamental intervals of the RCF in terms of its convergents \cite{Khi64}. Since this type of formula is not available for the SRCF, we use a bounded distortion argument.

Establishing the lower bound $\dim_{\rm H}G_\sigma(B,f)\geq\tau(B)/2$ 
is the most creative part of this paper, and hence
deserves a special attention with a close comparison to \cite{T}.
The construction of the fractal set in \cite{T} is rather rigid, which relies
on the ergodic theory of the Gauss map. 
Key
components are
a sequence of finite subsystems with increasing digits,
and an associated sequence of ergodic measures,
together with a finite collection of intervals which approximate the ergodic measure in a particular sense.
The intervals in the subsystems are glued together to form a fractal subset of $G(B,f)$.
The Hausdorff dimension of this set is estimated from below
in terms of measure-theoretic entropies and Lyapunov exponents of the ergodic measures.
To obtain stationary sequences
with ergodic arguments (Birkhoff's and Shannon-McMillan-Breiman's theorems),
a sufficiently high iteration of each subsystem is necessary. 
The number of this high iteration is not controllable, and
as a result, \cite{T} cannot be adapted to constructing a fractal subset of $G_\sigma(B,f)$ in the case where  
 $\sigma$ is not eventually constant. 
 
Our strategy for establishing the lower bound is to dispense with the ergodic arguments in \cite{T} altogether, and introduce a much more flexible 
 construction by exploiting
the theory of non-autonomous conformal
iterated function systems (IFSs) developed by Rempe-Gillen and Urba\'nski \cite{RU}. 
In Section~2 we provide preliminary results,
and in Section~3 prove the Main Theorem.
\section{Preliminaries}
In this section we provide preliminary results on the SRCF and non-autonomous conformal IFSs.
In Section~\ref{IFS} we explain a relationship between 
the SRCF in \eqref{CF-general} and an infinite IFS.
In Section~\ref{ncifs} we introduce non-autonomous conformal IFSs and recall the result of Rempe-Gillen and Urba\'nski \cite{RU}. 
In Section~\ref{BD} we establish extendibility and distortion bounds for IFSs generating the SRCF.

 \begin{figure}
\begin{center}
\includegraphics[height=4cm,width=5.5cm]{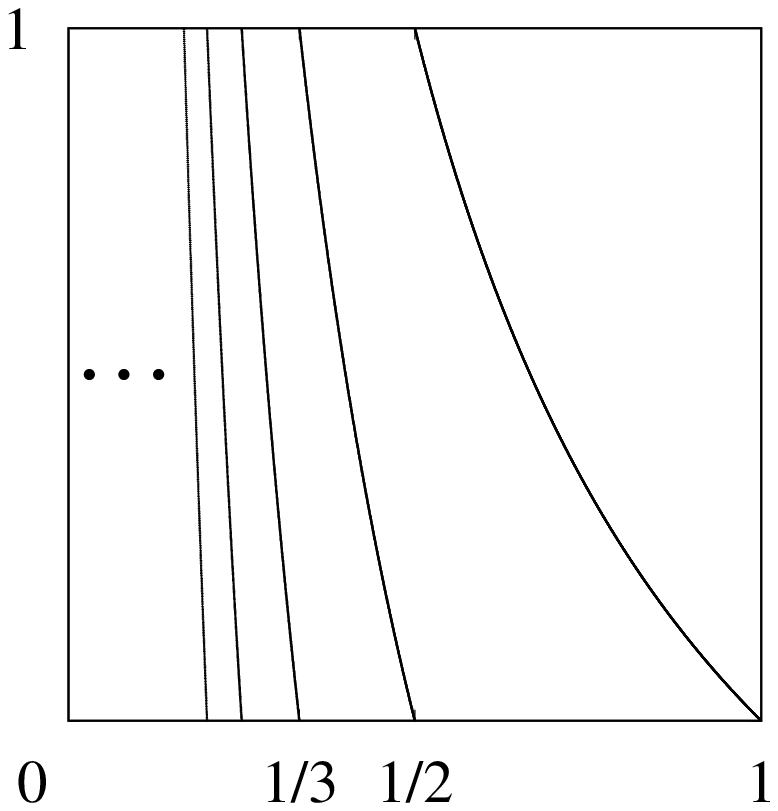}
\includegraphics[height=4cm,width=5.5cm]{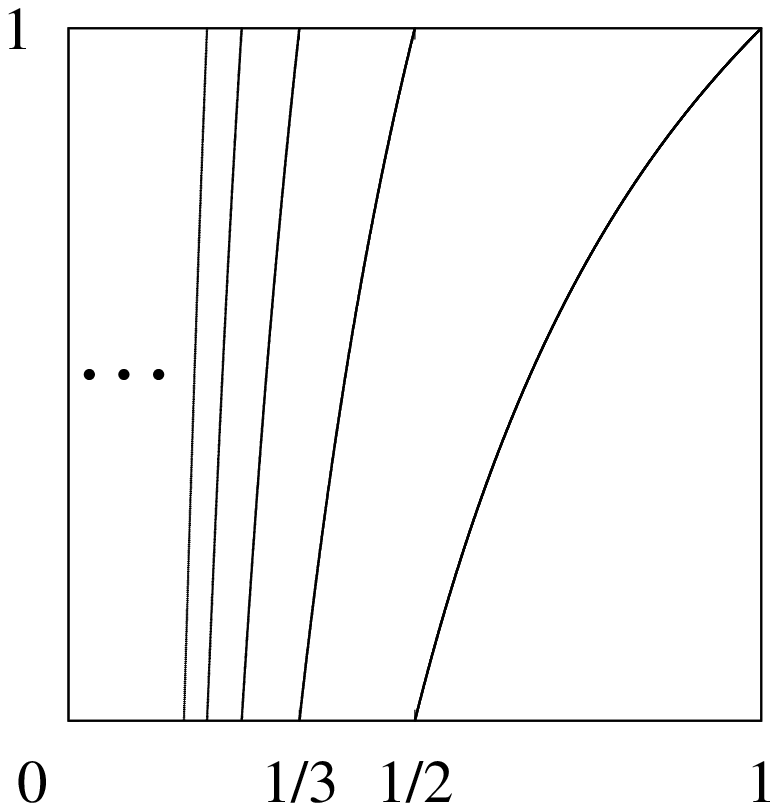}
\caption
{The graph of the Gauss map
$x\in(0,1]\mapsto1/x- \lfloor 1/x\rfloor\in[0,1)$ 
(left): that of the R\'enyi-like map
$x\in(0,1]\mapsto \lfloor1/x\rfloor-1/x+1\in(0,1]$ (right)}\label{GR}
\end{center}
\end{figure}

\subsection{The SRCF as an infinite IFS}\label{IFS}
Consider an infinite IFS $\{\varphi_{1,i}\}_{i=1}^\infty\cup\{\varphi_{-1,i}\}_{i=2}^\infty$ on $[0,1]$ defined by 
\[\varphi_{1, i}(x)=\frac{1}{i+x},\ x\in [0,1]\ \text{ and }\ \varphi_{-1, i}(x)=\frac{1}{i-x},\ x\in [0,1].\] 
These maps are obtained from inverse branches of the Gauss map $x\in(0,1]\mapsto1/x- \lfloor 1/x\rfloor\in[0,1)$ and that of the R\'enyi-like map 
$x\in(0,1]\mapsto \lfloor1/x\rfloor-1/x+1\in(0,1]$ respectively (see \textsc{Figure~\ref{GR}}).
The SRCF in \eqref{CF-general} is generated by this infinite IFS.
We have
\begin{equation}
\label{normal}
\varphi_{1,i}([0,1])=\left[\frac{1}{1+i}\ , \frac{1}{i}\right]\ \text{ and }\ \varphi_{-1,i}([0,1])=\left[\frac{1}{i}\ , \frac{1}{i-1}\right].
\end{equation}
For $\sigma=\{\sigma_n\}_{n=1}^\infty\in\{-1,1\}^\mathbb N$, $n\geq1$ and $a_1,\ldots,a_n\in\mathbb N$ such that
$\sigma_k+a_{k}\geq1$ for all $1\leq k\leq n$,
we define a fundamental interval
\begin{equation}\label{interval}[a_1,\ldots,a_n]_\sigma=\varphi_{\sigma_1, a_1}\circ \cdots \circ \varphi_{\sigma_n, a_n}([0,1]).\end{equation}
This is the closure of the interval which consists of points
having a (finite or infinite) SRCF  expansion beginning by $a_1,a_2,\ldots,a_n$.
Let
$|J|$ denote the Euclidean length of a bounded interval $J\subset\mathbb R$.

The next lemma states convergence properties of the SRCF \cite{Kra91,Per50,Tie11}
and a uniqueness property of the expansion.
It follows that
for each fixed $\sigma$, the SRCF in \eqref{CF-general} is generated by an infinite non-autonomous IFS whose limit set contains $\mathbb I$.

\begin{lem}\label{IFS}
\
\begin{itemize}
\item[(a)] Let $\sigma=\sigma_1\sigma_2\cdots\in\{-1,1\}^{\mathbb N}$
and $a_1a_2\cdots\in \mathbb N^{\mathbb N}$
be such that
 $\sigma_n+a_{n}\geq1$ for all $n\geq1$. The set
\[\bigcap_{n=1}^\infty[a_1,\ldots,a_n]_\sigma\]
is a singleton,
and its unique element equals the semi-regular continued fraction
\[
\confrac{1 }{a_{1}} + \confrac{\sigma_1 }{a_{2}}  + \confrac{\sigma_{2} }{a_{3} }+\cdots.
\]
This continued fraction is irrational 
if and only if $\sigma_n+a_{n}\geq2$ for infinitely many $n\geq1$.
\item[(b)]
 For any $x\in\mathbb I$ and any $\sigma\in\{-1,1\}^{\mathbb N}$,
there exists a unique $a_{1}a_{2}\cdots\in\mathbb N^{\mathbb N}$ such that
 $\sigma_n+a_{n}\geq1$ for all $n\geq1$
  and
  \[x\in\bigcap_{n=1}^\infty[a_1,\ldots,a_n]_\sigma.\]
 \end{itemize}
\end{lem}
\begin{proof}Let $\sigma=\sigma_1\sigma_2\cdots\in\{-1,1\}^{\mathbb N}$
and $a_{1}a_{2}\cdots\in \mathbb N^{\mathbb N}$
 satisfy
 $\sigma_n+a_{n}\geq1$ for all $n\geq1$.
By induction based on \eqref{normal} one can check that for all $n\geq2$
and all $x\in [0,1]$,
\begin{equation}\label{conv-lem}\varphi_{\sigma_1, a_{1}}\circ \cdots \circ \varphi_{\sigma_n, a_{n}}(x)=\confrac{1 }{a_{1}} + \confrac{\sigma_1 }{a_{2}}  +\cdots+ \confrac{\sigma_{n-1} }{a_{n}+\sigma_nx}.\end{equation}
If $\sigma_n+a_{n}\geq2$ for infinitely many $n\geq1$,
then the number in the right-hand side of \eqref{conv-lem} converges and the limit is  irrational \cite{Kra91,Tie11}.
Otherwise, it converges to a rational number
since the continued fraction
\[\confrac{1 }{2 } + \confrac{-1 }{2 }+\confrac{-1}{2}  +\cdots+ \confrac{-1 }{2 }+\cdots\]
is equal to $1$.
The proof of (a) is complete.

Let $x\in\mathbb I$ and let $\sigma\in\{-1,1\}^{\mathbb N}$. We can define inductively a sequence $a_{1},a_{2},\ldots$ in $\mathbb N$
such that
 $\sigma_n+a_{n}\geq1$ and
$x\in[a_{1},\ldots,a_{n}]_\sigma$ 
for all $n\geq1$.
If $b_1,b_2,\ldots$ is an integer sequence such that 
$[a_{1},\ldots,a_{n}]_\sigma=[b_{1},\ldots,b_{n}]_\sigma$ for all $n\geq1$, then
 \eqref{normal} and the irrationality of $x$ yield $a_n=b_n$ for all $n\geq1$.
 The proof of (b) is complete.
\end{proof}


\subsection{Non-autonomous conformal IFSs}\label{ncifs}
Below we summarize 
the dimension theory in \cite{RU} as far as we need them, on non-autonomous conformal IFS on the Euclidean space $\mathbb R^d$ of arbitrary dimension.
We note that $d=1$ throughout our application of their theory in Section~\ref{low-sec}.

Let $W\subset \mathbb{R}^d$ be an open set and let $\phi\colon W\rightarrow \phi(W)$ be a diffeomorphism. We say $\phi$ is {\it conformal} if for any $x\in W$ the differential $D\phi(x)\colon\mathbb{R}^d\rightarrow \mathbb{R}^d$ is a similarity linear map: 
$D\phi(x)=c_x\cdot M_x$ where $c_x>0$ is a scaling factor at $x$ and $M_x$ is a $d\times d$ orthogonal matrix. For a conformal map $\phi\colon W\to\phi(W)$
and a set $A\subset W$, we set
\[\Vert D\phi\Vert_A=\sup\{|D\phi(x)|\colon x\in A\},\]
where $|D\phi(x)|$ denotes the scaling factor of $\phi$ at $x$. 



For each $n\in \mathbb N$ let $I^{(n)}$
be a finite set. We introduce index sets
\begin{equation}\label{index-set}I^{\infty}=\prod_{j=1}^{\infty}I^{(j)}, \text{ and } I^k_n=\prod_{j=n}^k I^{(j)} \text{ for }k\geq n. 
\end{equation}
For each $n\in\mathbb N$
let $\{\phi_i^{(n)}\}_{i\in I^{(n)}}$ be a finite collection 
of self maps of a connected compact set $X\subset \mathbb{R}^d$
such that
 the closure of its interior $X^\circ$
coincides $X$. 
We assume $X$ is a convex set, or its boundary $\partial X$ is smooth. 
For $\omega=\omega_1\omega_2\cdots\in I^\infty$ and
 $n, k\in \mathbb{N}$ with $n\leq k$, we set \begin{equation}\label{indexII}\omega|_n^{k}=\omega_n\cdots \omega_{k}\in I_n^{k}\ \text{ and }\
\phi_{\omega|_n^k}=\phi_{\omega_n}^{(n)}\circ\cdots\circ \phi_{\omega_{k}}^{(k)}.\end{equation}

 A {\it non-autonomous conformal IFS} on $X$ is a sequence 
$\Phi=(\Phi^{(n)})_{n=1}^{\infty}$,
$\Phi^{(n)}=\{\phi_i^{(n)}\}_{i\in I^{(n)}}$
of collections of conformal self maps of $X$ 
which satisfies the following four conditions:
\begin{itemize}
\item[(A1)] (Open set condition) For all $n\in\mathbb N$ and all distinct indices $i,j\in I^{(n)},$ 
\[
\phi_i^{(n)}(X^{\circ})\cap \phi_j^{(n)}(X^{\circ})=\emptyset.
\]

\item[(A2)] (Conformality) There exists a connected open set $\tilde X$ of $\mathbb R^d$ containing $X$ 
such that each $\phi_i^{(n)}$ extends to a $C^{1}$ conformal diffeomorphism $\tilde{\phi}_i^{(n)}\colon\tilde X\to\tilde{\phi}_i^{(n)}(\tilde X)\subset\tilde X$.
\item[(A3)] (Bounded distortion) 
There exists $C\ge 1$ such that for all $\omega\in I^\infty$ and all $n$, $k\in \mathbb{N}$ with $n\leq k$, 
\[
|D\tilde\phi_{\omega|_n^k}(x_1)|\le C|D\tilde\phi_{\omega|_n^k}(x_2)|\ \text{ for all }x_1,x_2\in \tilde X.
\]

\item[(A4)] (Uniform contraction) 
There are constants $0< \gamma <1$ and $L\geq1$ such that for all $\omega\in I^{\infty}$ and all $n,k\in\mathbb N$ with
$k-n\geq L$, 
\[
\|D\phi_{\omega|_n^k}\|_{X}\le\gamma^{k-n+1}.
\]
\end{itemize}

If $\Phi=(\Phi^{(n)})_{n=1}^{\infty}$ is a non-autonomous conformal IFS,
then (A4) ensures that the set $\bigcap_{n=1}^{\infty}\phi_{\omega|_{1}^n}(X)$ is a singleton for each $\omega\in I^{\infty}$.
We define an
 address map $\Pi \colon I^{\infty} \to X$ by
\[\Pi(\omega)\in \bigcap_{n=1}^{\infty}\phi_{\omega|_{1}^n}(X),\]
and the limit set of $\Phi$ by
 \[
\Lambda(\Phi)=\Pi(I^{\infty}).
\]

In order to determine the Hausdorff dimension of the limit set, for $s\geq0$ 
we introduce a partition function
\[
Z_n^{\Phi}(s)=\sum_{\omega\in I_{1}^n}(\Vert D\phi_{\omega}\Vert_{X})^s,
\]
and
a lower pressure function
$\underline{P}^\Phi\colon[0,\infty)\to [-\infty,\infty]$ of $\Phi$ by
\[\underline{P}^{\Phi}(s)=\liminf\limits_{n\rightarrow \infty}\frac{1}{n}\log Z_n^{\Phi}(s).\]
The lower pressure function has the following monotonicity
\cite[Lemma~2.6]{RU}:
if $0\leq s_1<s_2$ then $\underline{P}^\Phi(s_1)=\underline{P}^\Phi(s_2)=\infty$ or
$\underline{P}^\Phi(s_1)=\underline{P}^\Phi(s_2)=-\infty$ or $\underline{P}^\Phi(s_1)>\underline{P}^\Phi(s_2)$. So, one can  define a critical value
\[s(\Phi)={\rm sup}\{s\ge0\colon  \underline{P}^\Phi(s)>0\}={\rm inf}\{s\ge0\colon \underline{P}^\Phi(s)<0\},
\]
called the Bowen dimension. Note that  (A1) and (A3) together imply $s(\Phi)\le d$.
We say the non-autonomous conformal IFS $\Phi$ is {\it subexponentially bounded} if
\[
\lim_{n \to \infty}\frac{1}{n}\log \# I^{(n)}=0.
\]
We will use the following result in Section~\ref{low-sec}.

\begin{thm}[Bowen's formula, {\cite[Theorem~1.1]{RU}}]
\label{Bowen}
Let $\Phi$ be a non-autonomous conformal IFS that is subexponentially bounded.
Then  
\[\dim_{\rm H}\Lambda(\Phi)=s(\Phi).
\]
\end{thm}

\subsection{An extendibility and bounded distortion}
\label{BD}
Below we show that each map in the infinite IFS 
$\{\varphi_{1,i}\}_{i=1}^\infty\cup\{\varphi_{-1,i}\}_{i=2}^\infty$ admits a $C^2$ conformal extension to the interval
\[\tilde X=\left(-\frac{1}{5}, \frac{5}{4}\right)\supset [0,1].\] 
\begin{lem}
\label{conf}
\
\begin{itemize}
\item[(a)] For each $i\geq1$, $\varphi_{1, i}$ and
$\varphi_{-1, i+1}$ extend to $C^{2}$ diffeomorphisms  $\tilde\varphi_{1,i}\colon\tilde X\to \tilde\varphi_{1,i}(\tilde X)\subset\tilde X$ and
$\tilde\varphi_{-1,i+1}\colon\tilde X\to \tilde\varphi_{-1,i+1}(\tilde X)\subset\tilde X$
respectively.
\item[(b)] For all $i\geq3$ we have
 \[\Vert D\tilde\varphi_{1, i}\Vert_{\tilde X}\le \frac{1}{2}\ \text{ and }\ \Vert D\tilde\varphi_{-1,i}\Vert_{\tilde X}\le \frac{1}{2}.\]
 \end{itemize}
\end{lem}
\begin{proof}
We set 
$\tilde\varphi_{1, i}(x)=1/(i+x),\ x\in \tilde X$ for $i\geq1$ and $\tilde\varphi_{-1, i}(x)=1/(i-x),\ x\in\tilde X$ for $i\geq3$.
Then $\tilde\varphi_{1,i}$ and
$\tilde\varphi_{-1,i}$ are $C^2$ diffeomorphisms onto their images and satisfy \[\tilde{\varphi}_{1, i}(\tilde X)=\left(\frac{1}{i+5/4}, \frac{1}{i-1/{5}}\right)\subset \tilde X\text{ and }\]  \[\tilde{\varphi}_{-1, i}(\tilde X)=\left(\frac{1}{i+1/{5}}, \frac{1}{i-5/4}\right)\subset \tilde X.\]

The remaining map $\varphi_{-1,2}$ needs a different treatment since it has $1$  as a neutral fixed point. 
Let $\tilde\varphi_{-1,2}\colon \tilde X\to\mathbb R$
be a $C^2$ diffeomorphism onto its image
such that 
 $\tilde\varphi_{-1,2}(x)=1/(2-x)$
for $x\in(-1/{5},1]$ and $\tilde\varphi_{-1,2}(5/4)<5/4$. 
 Then $\tilde\varphi_{-1,2}(\tilde X)\subset \tilde X$ holds. We have verified (a). Item (b) follows from direct calculations.
\end{proof}

For the rest of this paper, we denote the $C^2$ conformal extensions $\tilde\varphi_{1, i}$, $\tilde \varphi_{-1, i}$
in Lemma~\ref{conf} by $\varphi_{1, i}$, $\varphi_{-1, i}$ for ease of notation. 
  We now establish a bounded distortion property
  for these extensions.
For $L>1$ put $\mathbb N_{\geq L}=\{n\in\mathbb N\colon n\geq L\}$.


\begin{lem}\label{BDP}
There exists $C\ge1$ such that for all
$\sigma\in\{-1,1\}^\mathbb N$, 
 all $a_1a_2\cdots\in ({\mathbb N_{\geq3}})^\mathbb N,$ all $k\in \mathbb{N}$ and for  all $x, y\in \tilde X$ we have \[|D(\varphi_{\sigma_1, a_1}\circ \cdots \circ \varphi_{\sigma_k, a_k})(x)|\le C|D(\varphi_{\sigma_1, a_1}\circ \cdots \circ \varphi_{\sigma_k, a_k})(y)|.\]
\end{lem}


\begin{proof}
By Lemma~\ref{conf}(b),
for $\sigma\in \{-1, 1\}^\mathbb N$,  $a_1a_2\cdots\in (\mathbb N_{\geq3})^\mathbb N$, $k\in \mathbb{N}$,
$1\leq j\leq k$ and
for $x,y\in \tilde X$ we have
\[|\varphi_{\sigma_j, a_j}\circ \cdots \circ \varphi_{\sigma_k, a_k}(x)-\varphi_{\sigma_j, a_j}\circ \cdots \circ \varphi_{\sigma_k, a_k}(y)|\leq\left(\frac{1}{2}\right)^{k-j}|\tilde X|.\]
A direct calculation shows that there is a constant $C_0>0$ such that 
$\Vert D\log|D\varphi_{1,i}|\Vert\leq C_0$ and $\Vert  D\log|D\varphi_{-1,i}|\Vert\leq C_0$
 for all $i\in
{\mathbb N_{\geq3}}$.
Therefore we obtain

\[\begin{split}
\log\frac{|D(\varphi_{\sigma_1, a_1}\circ \cdots \circ \varphi_{\sigma_k, a_k})(x)|}{|D(\varphi_{\sigma_1, a_1}\circ \cdots \circ \varphi_{\sigma_k, a_k})(y)|}&\leq 
\sum_{j=1}^kC_0\left(\frac{1}{2}\right)^{k-j}|\tilde X|\le 2C_0|\tilde X|.
\end{split}\]
The desired inequality holds
if we set $C=\exp(2C_0|\tilde X|)$.
\end{proof}

\section{The Hausdorff dimension}
Using the preliminary results in Section~2, we prove main lower and upper bounds of Hausdorff dimension
 in Sections~\ref{low-sec} and 
 \ref{up-sec} respectively.
  In Section~\ref{pfthms} we combine both bounds and complete the proofs of the Main Theorem.

\subsection{The lower bound}\label{low-sec}
Using the theory of non-autonomous conformal IFSs summarized in Section~\ref{ncifs}, 
we prove
the following lower bound of Hausdorff dimension.
\begin{pro}
\label{lower}
Let $\sigma\in \{-1, 1\}^{\mathbb{N}},$ let $B$ be an infinite subset of $\mathbb{N}$ and let $f\colon\mathbb{N}\rightarrow [\min B, \infty)$ satisfy $\lim_{n\to \infty}f(n)=\infty.$ 
Then \[\dim_{\rm H}G_{\sigma}(B, f)\ge \frac{\tau(B)}{2}.\]
\end{pro}
\begin{proof}
We may assume $\tau(B)>0$, for otherwise there is nothing to prove.
Let $\varepsilon\in (0, \tau(B))$. As in \cite{T}, we define a strictly increasing sequence $\{b_m\}_{m\in \mathbb{N}}$ in $B$ inductively as follows: $b_1=\min B,$ and $b_{m+1} > b_m$ is the minimal in $B$ such that
\begin{equation}
\label{aaaa}
\sum_{k\in B\cap[b_m,b_{m+1})}k^{-\tau(B)+\varepsilon}\ge 1. 
\end{equation}
Set 
\begin{equation}
\label{defbm}
 B_m=
\begin{cases}
 \{\min B\}&\text{for } m=1,\\
B\cap[b_m,b_{m+1})&\text{for }m\ge 2.
 \end{cases}
 \end{equation}
Note that $B=\bigcup_{m=1}^\infty B_m.$
Let $\{t_m\}_{m\in \mathbb{N}}$ be a sequence of positive integers such that for every $m\ge 2$ we have
\begin{equation}
\label{hutou}
b_{m+1}\le \inf\left\{f(n)\colon \sum_{j=1}^{m-1}t_j+1\le n \le  \sum_{j=1}^{m}t_j\right\},
\end{equation}
and
\begin{equation}
\label{hutou2}
\lim_{m\to \infty}\frac{\log \# B_m}{ \sum_{j=1}^{m}t_j}=\lim_{m\to \infty}\frac{\log \# B_m}{ \sum_{j=1}^{m-1}t_j+1}=0.
\end{equation}
Since $\lim_{n\to \infty} f(n) = \infty$, one can choose such a sequence by induction on $m$. 

Put $X=[0,1]$. We now construct a non-autonomous conformal IFS on $X$. 
For each integer $1\leq n\leq t_1$, we set
\[I^{(n)}=B_1.\]
For each 
integer $n\geq t_1+1$ we pick $m\ge 2$ such that $\sum_{j=1}^{m-1}t_j+1\le n \le  \sum_{j=1}^{m}t_j$,
and set \[I^{(n)}=B_m.\] Put 
$I^{\infty}=\prod_{j=1}^{\infty}I^{(j)}$ and
$I^k_n=\prod_{j=n}^k I^{(j)}$
as in \eqref{index-set}. 
For $n\geq t_1+1$ we have \[I_1^n=(B_1)^{t_1}\times \cdots \times (B_{m-1})^{t_{m-1}}\times 
(B_m)^{l_n},\]
where $l_n=n-\sum_{j=1}^{m-1}t_j$ and  
\[I^{\infty}=(B_1)^{t_1}\times (B_2)^{t_2}\times \cdots \times (B_m)^{t_m}\times \cdots.\] 
Let $n\in\mathbb N$. For each $i\in I^{(n)}$ we put
$\phi_i^{(n)}=\varphi_{\sigma_n,i},$
and set
$\Phi^{(n)}=\{\phi^{(n)}_i\}_{i\in I^{(n)}}.$ 
For $\omega=\omega_1\omega_2\cdots\in I^\infty$ and
 $n, k\in \mathbb{N}$ with $n\leq k$, we set $\omega|_n^{k}=\omega_n\cdots \omega_{k}\in I_n^{k}$ and 
$\phi_{\omega|_n^k}=\phi_{\omega_n}^{(n)}\circ\cdots\circ \phi_{\omega_{k}}^{(k)}$
as in \eqref{indexII}.
Finally we set
$\Phi=\{\Phi^{(n)}\}_{n=1}^\infty.$ 
Then (A1) follows from \eqref{normal}, (A2) follows from 
 Lemma~\ref{conf}, and (A3) follows from Lemma~\ref{BDP}.
Since 
$\|D \phi^{(n)}_i\|_X=\|D \varphi_{\sigma_n, i}\|_X\le 1/4$
for all $n\ge t_1+t_2+1$ and all $i\in I^{(n)},$ (A4) holds. Therefore $\Phi$ is a non-autonomous conformal IFS on $X$.






\begin{lem}\label{contained}
The limit set $\Lambda(\Phi)$ of $\Phi$ is contained in $G_{\sigma}(B, f)$.
\end{lem}
\begin{proof}
Let $\Pi$ denote the address map of the non-autonomous conformal IFS $\Phi$.
Let $\omega\in I^{\infty}$. Then 
$a_{\sigma, n}(\Pi(\omega))\in B$ for all $n\geq1$. The first alternative of \eqref{defbm} gives \[a_{\sigma, n}(\Pi(\omega))=\min B\le f(n)\ \mbox{ if }1\leq n\leq t_1.\] For every $m\ge 2$, the second alternative of \eqref{defbm} and \eqref{hutou} yield 
\[b_m\le a_{\sigma, n}(\Pi(\omega))<b_{m+1}\le f(n)\ \mbox{ if } \sum_{j=1}^{m-1}t_j+1\leq n\leq \sum_{j=1}^{m}t_j.\] 
As $n\to\infty$ we have $m\to\infty$, $b_m\to\infty$ and so $a_{\sigma, n}(\Pi(\omega))\to \infty$.
Hence $\Pi(\omega)\in G_{\sigma}(B, f)$ holds.
Since $\omega\in I^\infty$ is arbitrary we obtain $\Lambda(\Phi)\subset G_{\sigma}(B, f)
$. \end{proof}
\begin{lem}\label{subexp}The non-autonomous conformal IFS $\Phi$
is subexponentially bounded.
\end{lem}
\begin{proof}Recall that for each integer $n\geq t_1+1$ we have 
 $\sum_{j=1}^{m-1}t_j+1\le n $.
Then
\[0\leq \frac{1}{n}\log \# I^{(n)}\le \frac{\log \# B_m}{\sum_{j=1}^{m-1}t_j+1}.\]
From this and \eqref{hutou2} we obtain
$\lim_{n\to \infty}(1/n)\log\# I^{(n)}=0$.
\end{proof}

Recall that the lower pressure function $\underline{P}^\Phi\colon[0,\infty)\to[-\infty,\infty]$ 
  is
 given by \[\underline{P}^\Phi(s)=\liminf\limits_{n\rightarrow \infty}\frac{1}{n}\log Z_n^\Phi(s),\text{ where }
 Z_n^\Phi(s)=\sum_{\omega\in I_1^n}(\Vert D\phi_{\omega}\Vert_X)^s.\]
 By Lemmas~\ref{contained}, \ref{subexp}
and Theorem~\ref{Bowen},
the Bowen dimension $s(\Phi)$ satisfies 
\begin{equation}\label{eqp-1}\dim_{\rm H}G_\sigma(B,f)\geq\dim_{\rm H}\Lambda(\Phi)=
s(\Phi).\end{equation}
In order to estimate the Bowen dimension from below, we estimate the lower pressure function from below.
From Lemma~\ref{BDP}, there exists a constant $C>1$ such that for all $n\geq1$, all $i\in I^{(n)}$ and for all $x\in X$ we have
\begin{equation}
\label{control}
|D\phi_{i}^{(n)}(x)| \ge C^{-1}|D\phi_{i}^{(n)}(0)|=C^{-1}\frac{1}{k^2},
\end{equation}
where $k\geq1$ is the integer such that $\phi_i^{(n)}\in\{\varphi_{1,k},\varphi_{-1,k}\}$.
Since $b_m\to\infty$ as $m\to\infty$, for any $\delta>0$ there exists $N\geq1$ such that for all $i\geq b_{N+1}$ we have
\begin{equation}\label{control2}
C^{-1}i^{-2}\geq i^{-2-\delta}.\end{equation}
Using the chain rule,  \eqref{control} and \eqref{control2} 
we have
\[\begin{split}
Z_n^\Phi(s)=&
\sum_{\omega\in I_1^n}(\Vert D(\varphi_{\sigma_1,\omega_1}\circ \varphi_{\sigma_2,\omega_2}\circ\cdots \circ \varphi_{\sigma_n,\omega_n})\Vert_X)^s\\\ge& 
C^{-Ns}\sum_{\omega\in I_1^n}{\omega_1}^{-2s}\cdots
{\omega_{t_1+\cdots +t_N}}^{-2s}{\omega_{t_1+\cdots +t_N+1}}^{-(2+\delta)s}\cdots{\omega_n}^{-(2+\delta)s}
\\=& 
C^{-Ns}(\min B)^{-2st_1}\left(\sum_{k\in B_2}k^{-2s}\right)^{t_2}\cdots \left(\sum_{k\in B_{N}}k^{-2s}\right)^{t_{N}}\\
&\times\left(\sum_{k\in B_{N+1}}k^{-(2+\delta)s}\right)^{t_{N+1}}\cdots\left(\sum_{k\in B_m}k^{-(2+\delta)s}\right)^{l_n}
.
\end{split}\]
Set
$s_{\varepsilon,\delta}=(\tau(B)-\varepsilon)/(2+\delta)$. By \eqref{aaaa}
we have \[\left(\sum_{k\in B_{N+1}}k^{-(2+\delta)s_{\varepsilon,\delta}}\right)^{t_{N+1}}\cdots \left(\sum_{k\in B_m}k^{-(2+\delta)s_{\varepsilon,\delta}}\right)^{l_n}\geq1.\]
Substituting $s=s_{\varepsilon,\delta}$ and combining the above two inequalities yield
\[Z_n^\Phi(s_{\varepsilon,\delta})\ge C^{-Ns_{\varepsilon,\delta}}(\min B)^{-2s_{\varepsilon,\delta}t_1}\left(\sum_{k\in B_2}k^{-2s_{\varepsilon,\delta}}\right)^{t_2}\cdots \left(\sum_{k\in B_{N}}k^{-2s_{\varepsilon,\delta}}\right)^{t_{N}}.\]
It follows that $\underline{P}^\Phi(s_{\varepsilon,\delta})=\liminf_{n\to\infty}(1/n)\log Z_n^\Phi(s_{\varepsilon,\delta})\geq0$, and hence 
\begin{equation}\label{eqp-2}
s(\Phi)\ge \frac{\tau(B)-\varepsilon}{2+\delta}.
\end{equation}
Combining \eqref{eqp-1} and \eqref{eqp-2}, and 
 letting $\delta\to0$ and then $\varepsilon\to0$ yields
the desired inequality 
in Proposition~\ref{lower}.
\end{proof}

\subsection{The upper bound}\label{up-sec}
We slightly modify the argument by Hirst \cite[Corollary~1]{Hir73}
and prove the following upper bound of Hausdorff dimension.
\begin{pro}\label{upper}
Let $\sigma\in \{-1, 1\}^{\mathbb{N}}$ and let $B$ be an infinite subset of $\mathbb N$. We have \[\dim_{\rm H}E_{\sigma}(B)\le \frac{\tau(B)}{2}.\]
\end{pro}
\begin{proof}
Let $\varepsilon>0$. 
Let $L=L(\varepsilon)\geq3$
be a sufficiently large integer such that
\begin{equation}\label{upper-eq1}(2C)^{\frac{1}{2}(\tau(B)+\varepsilon) }\sum_{k\in B\cap\mathbb N_{\geq L}}\frac{1}{k^{\tau(B)+\varepsilon}}\leq1,\end{equation}
where $C>1$ is the distortion constant in Lemma~\ref{BDP}.
We construct a family of coverings of the set
\[E_{\sigma,L}(B)=\{x\in E_\sigma(B)\colon a_{\sigma,n}(x)\geq L\text{ for all } n\in\mathbb N\},\]
by fundamental intervals 
$[a_1,\ldots,a_n]_\sigma$ satisfying $a_1\cdots a_n\in (B\cap\mathbb N_{\geq L} )^n$.
Lemma~\ref{BDP} gives
\[\begin{split}\frac{|[a_1,\ldots,a_n,a_{n+1}]_\sigma|}{|[a_1,\ldots,a_n]_\sigma|}&\leq C\frac{|(\varphi_{\sigma_1, a_1}\circ \cdots \circ \varphi_{\sigma_n, a_n})^{-1}([a_1,\ldots,a_n,a_{n+1}]_\sigma )|}{|(\varphi_{\sigma_1, a_1}\circ \cdots \circ \varphi_{\sigma_n, a_n})^{-1}([a_1,\ldots,a_n]_\sigma)|}\\
&=C|\varphi_{\sigma_{n+1},a_{n+1}}([0,1])|\leq \max\{1,1-\sigma_{n+1}\}\frac{C}{a_{n+1}^2}\leq\frac{2C}{a_{n+1}^2}.\end{split}\]
Then we have
\[\sum_{a_{n+1}\in B\cap\mathbb N_{\geq L} }\frac{|[a_1,\ldots,a_n,a_{n+1}]_\sigma|^{\frac{1}{2}(\tau(B)+\varepsilon) }}{|[a_1,\ldots,a_n]_\sigma|^{\frac{1}{2}(\tau(B)+\varepsilon) }}\leq (2C)^{\frac{1}{2}(\tau(B)+\varepsilon) }\sum_{a_{n+1}\in B\cap\mathbb N_{\geq L} }\frac{1}{a_{n+1}^{\tau(B)+\varepsilon}},\]
which does not exceed $1$ by \eqref{upper-eq1}.
It follows that 
\[\frac{\sum_{a_1\cdots a_{n+1}\in (B\cap\mathbb N_{\geq L} )^{n+1}}|[a_1,\ldots,a_{n+1}]_\sigma|^{\frac{1}{2}(\tau(B)+\varepsilon) }}{ \sum_{a_1\cdots a_n\in (B\cap\mathbb N_{\geq L} )^n}|[a_1,\ldots,a_{n}]_\sigma|^{\frac{1}{2}(\tau(B)+\varepsilon) }}\leq1,\]
and therefore
\[\sum_{a_1\cdots a_{n+1}\in (B\cap\mathbb N_{\geq L} )^{n+1}}|[a_1,\ldots,a_{n+1}]_\sigma|^{\frac{1}{2}(\tau(B)+\varepsilon) }\leq1.\]
Since $\sup\{|[a_1,\ldots,a_n]_\sigma|\colon a_1\cdots a_{n}\in (B\cap\mathbb N_{\geq L} )^n\}\to0$ as $n\to\infty$, we obtain
$\dim_{\rm H} E_{\sigma,L}(B)\leq(\tau(B)+\varepsilon)/2.$

The set $E_\sigma(B)$ is covered by the images of sets of the form $E_{\sigma',L}(B)$, $\sigma'=\sigma_j\sigma_{j+1}\cdots\in\{-1,1\}^\mathbb N$, $j\geq1$ under the infinite IFS $\{\varphi_{1,i}\}_{i=1}^\infty\cup\{\varphi_{-1,i}\}_{i=2}^\infty$.
Since $\varphi_{1,i}$, $\varphi_{-1,i}$ are $C^1$ diffeomorphisms and Hausdorff dimensions do not change under the action of bi-Lipschitz homeomorphisms, we have
$\dim_{\rm H} E_\sigma(B)\leq(\tau(B)+\varepsilon)/2$. Since $\varepsilon>0$ is arbitrary we obtain the desired inequality.
\end{proof}

\subsection{Proof of the Main Theorem}\label{pfthms}
Let $B$ be an infinite subset of $\mathbb N$ and let $f\colon\mathbb N\to[\min B,\infty)$ satisfy $\lim_{n\to\infty}f(n)=\infty$.
Since $E_\sigma(B)\supset G_\sigma(B,f)$ holds for all $\sigma\in\{-1,1\}^{\mathbb N}$,
Combining Propositions~\ref{lower} and \ref{upper}  together yields $\dim_{\rm H} E_\sigma(B)=\dim_{\rm H}G_\sigma(B,f)
=\tau(B)/2$.\qed

\subsection*{Acknowledgments}
We thank the referee for his/her careful reading of the manuscript and valuable suggestions. We thank Gerardo Gonz\'alez Robert, Mumtaz Hussain, Nikita Shulga for fruitful discussions.
The second named author was partially supported by the JSPS KAKENHI 
23K20220. 
\end{document}